\documentclass[preprint,11pt,reqno]{article}

\nonstopmode

\usepackage[left=2cm,right=2cm,top=2cm,bottom=2cm]{geometry}
\usepackage[utf8]{inputenc}
\usepackage[english]{babel}
\usepackage[T1]{fontenc}
\usepackage{amsmath}
\usepackage{amsthm}
\usepackage{amsfonts}
\usepackage{amssymb}
\usepackage{color}
\usepackage{multicol}
\usepackage{hyperref}
\hypersetup{
colorlinks=true, 
breaklinks=true, 
urlcolor= red, 
linkcolor= red, 
citecolor=red, 
}
\usepackage{graphicx}
\usepackage{lscape}
\usepackage{setspace}

\newcommand{\norm}[1]{\left\lVert #1 \right\rVert}

\newtheorem{theorem}{Theorem}[section]
\newtheorem{lemma}{Lemma}[section]
\newtheorem{proposition}{Proposition}[section]

\newtheorem{definition}{Definition}[section]
\newtheorem{remark}{Remark}[section]
\title{On solutions of the Diophantine equation $L_n+L_m=3^a$ } 

\author{Pagdame TIEBEKABE $\&$ Ismaïla DIOUF\\}
\date{\today}
\begin{document}
\maketitle
\begin{abstract}
Let $(L_n)_{n\geq 0}$ be the Lucas sequence given by $L_0 = 2, L_1 = 1$ and $L_{n+2} = L_{n+1}+L_n$ for $n \geq 0$. In this paper, we are interested in finding all powers of three which are sums of two Lucas numbers, i.e., we study the exponential Diophantine equation $L_n + L_m = 3^{a}$ in nonnegative integers $n, m,$ and $a$. The proof of our main theorem uses lower bounds for linear forms in logarithms, properties of continued fractions, and a version of the Baker-Davenport reduction method in Diophantine approximation.
\end{abstract}
\textbf{Keywords and phrases}: Linear forms in logarithm; Diophantine equations; Fibonacci sequence; Lucas sequence; perfect powers.\\
\textbf{2010 Mathematics Subject Classification.} 11B39, 11J86

\section{Introduction} 

The determination of perfect powers of Lucas and Fibonacci sequences does not date from today.
The real contribution of determination of perfect powers of Lucas and Fibonacci sequences began in 2006.
By classical and modular approaches of Diophantine equations, Bugeaud,  Mignotte, and Siksek \cite{5} defined all  perfect powers of Lucas and Fibonacci sequences by solving the equations $F_n=y^p$ and $L_n=y^p$ respectively. From there, many researchers tackled similar problems.
It is in the same thought that, others have determined the powers of $2$ of the sum/difference of two Lucas numbers \cite{3}, powers of $2$ of the sum/difference of Fibonacci numbers \cite{4}, powers of $2$ and of $3$ of the product of Pell numbers and Fibonacci numbers.

We move our interest on the powers of $3$ as a  sum of two Lucas numbers.
This paper follows the following steps :
We first give the generalities on binary linear recurrence, then we demonstrate an important inequality on Lucas numbers and finally determine and reduce a coarse bound by section $3$.
The section $4$ is devoted to the reduction of the obtained bound in section $3$ and discussion of possible different cases.
We know from   Bravo and Lucas \cite{3} that the only solutions of the Diophantine equation $F_n + F_m = 2^a$ in positive integers $n$, $m$ and $a$ with $n \geq m$ are given by
$$
2F_1=2, \quad 2F_2=2, \quad 2F_3=4, \quad 2F_6=16,$$ and $$F_2+F_1=2, \quad F_4+F_1=F_4+F_2=4, \quad F_5+F_4=8,\quad F_7+F_4=16.
$$
 and in \cite{4} that
all solutions of the Diophantine equation $L_n+L_m=2^{ a }$ in nonnegative integers $n \geq m$
and $a$, are 
$$
2L_0=4,\quad 2L_1=2,\quad 2L_3=8,\quad L_2+L_1=4, \quad L_4+L_1=8, \quad \text{and}\quad L_7+L_2=32.
$$

Here in this paper, we determine all the solutions of the following Diophantine equation:
\begin{equation}\label{eq:1}
L_n+L_m=3^a
\end{equation}
 in nonnegative integers $n\geq m$ and $a$.
 
  We are interested in finding all powers of three which are sums of two Lucas numbers, i.e., we study the exponential Diophantine equation $L_n + L_m = 3^{a}$ in nonnegative integers
 $n$, $ m,$ and~$a$. The proof of our main theorem uses lower bounds for linear forms in logarithms, properties of continued fractions, and a version of the Baker-Davenport reduction method in Diophantine approximation.

We  notice that  many authors have already tackled this type of problems.

\section{Preliminaries}

\subsection{Generalities}

\begin{definition}

Let $k\geq 1$. The sequence $\{H_n\}_{n\geq 0}\subseteq \mathbb{C}$ is called a recurrent linear sequence of order $k$ if the sequence satisfies
\begin{equation*}
H_{n+k}=a_1H_{n+k-1}+a_2H_{n+k-2}+ \cdots+a_kH_n
\end{equation*}
for all $n\geq 0$ with $a_1,\ldots,a_k\in \mathbb{C}$, fixed.
\end{definition}

We suppose that $a_k\neq 0$ (otherwise, the sequence $\{H_n\}_{n\geq 0}$ satisfies a recurrence of order less than $k$). If $a_1,\ldots ,a_k \in \mathbb{Z}$ and $H_0,\ldots,H_{k-1}\in \mathbb{Z}$, then we can easily prove by induction on $n$ that $H_n$ is an integer for all $n\geq 0$. The polynomial
\begin{equation*}
f(X)=X^k-a_1X^{k-1}-a_2X^{k-2}-\cdots -a_k \in \mathbb{C},
\end{equation*}
is called the characteristic polynomial of  $(H_n)_{n\geq 0}$. We suppose that
\begin{equation*}
f(X)=\prod_{i=1}^m(X-\alpha_i)^{\sigma_i},
\end{equation*}
where $\alpha_1$, \dots, $\alpha_m$ are distinct roots of $f(X)$ with respectively $\sigma_1$, \dots, $\sigma_m$ their multiplicities.

\begin{definition}

 We define the sequences  $(A_n)_{n\geq 0}$ and $(B_n)_{n\geq 0}$ for all positive integers  $\mathbb{N}$ by 
$$
\left\{
    \begin{array}{lllll}
        A_{n+2}& =& aA_{n+1}+A_n,&A_0=0,&A_1=1\\
        B_{n+2}& =& aB_{n+1}+B_n,&B_0=2,&B_1=a.\\
    \end{array}
\right.
$$

For $a=1$, $(A_n)_{n\geq 0}=(F_n)_{n\geq 0}$ and $(B_n)_{n\geq 0}=(L_n)_{n\geq 0}$ , which are  Fibonacci and Lucas sequences respectively, defined above.
\end{definition}

\begin{remark}

 If $k=2$, the sequence $(H_n)_{n\geq 0}$ is called a binary recurrent sequence. In this case, the characteristic polynomial is of the form 
\begin{equation*}
f(X)=X^2-a_1X-a_2=(X-\alpha_1)(X-\alpha_2).
\end{equation*}
\end{remark}

Suppose that $\alpha_1\neq \alpha_2$, then $H_n=c_1\alpha_1^n+c_2\alpha_2^n$
for all $n\geq 0$.
\begin{definition}

 The binary recurrent sequence  $\{H_n\}_{n\geq 0}$ is said to be non degenerated if $c_1c_2\alpha_1\alpha_2\neq 0$ and $\alpha_1/\alpha_2$ is not a root of unity.
 
\end{definition}

Binet's formula for the general term of Fibonacci and Lucas sequences is obtained using standard methods for solving recurrent sequences, which are given by :

\begin{equation*}
F_n=\dfrac{\alpha^n-\beta^n}{\alpha-\beta}\quad \text{and}\quad L_n=\alpha^n+\beta^n
\end{equation*}

where $\left(\alpha,\beta\right)=\left(\dfrac{1+\sqrt{5}}{2},\dfrac{1-\sqrt{5}}{2} \right)$ are the zeros of the characteristic polynomial $X^2-X-1$.

\begin{definition}

 For all algebraic numbers  $\gamma$, we define its measure by the following identity :
\begin{equation*}
{\rm M}(\gamma)=|a_d|\prod\limits_{i=1}^d \max \{1,|\gamma_{i}|\},
\end{equation*}
where $\gamma_{i}$ are the roots of $f(x)=a_d\prod\limits_{i=1}^d(x-\gamma_{i})$ is the minimal polynomial of  $\gamma$.
\end{definition} 
 
Let us define now another height, deduced from the last one, called the absolute logarithmic height. It is the most used one.

\begin{definition}( Absolute logarithmic height)
  
For a non-zero algebraic number of degree $d$ on $\mathbb{Q}$ where the minimal polynomial on  $\mathbb{Z}$ is $f(x)=a_d\prod\limits_{i=1}^d(x-\gamma_{i})$, we denote by
\begin{equation*}
h(\gamma)=\dfrac{1}{d}\left(\log|a_d|+\sum\limits_{i=1}^d \log\max\{1,|\gamma_i|\}\right) = \dfrac{1}{d} \log {\rm M}(\gamma).
\end{equation*}
the usual logarithmic absolute height of $\gamma$.
 \end{definition}

 The following properties of the logarithmic height, will also be used in the next section:
 \begin{itemize}
 \item[•] $h(\gamma \pm \eta)\leq h(\gamma)+h(\eta)+\log 2$.
 \item[•] $h(\gamma\eta^{\pm 1})\leq h(\gamma)+h(\eta)$.
 \item[•] $h(\gamma^s)=|s|h(\gamma)$.
 \end{itemize}

\subsection{Inequalities involving the Lucas numbers}

In this section, we state and prove important inequalities associated with the Lucas numbers that will be used in solving the equation (\ref{eq:1}).

\begin{proposition}
For $n\ge 2$, we have 
 \begin{equation}\label{eq:2}
 0.94 \, \alpha^n< (1-\alpha^{-6})\alpha^n \le L_n\le (1+\alpha^{-4})\alpha^n<1.15\, \alpha^n
 \end{equation}

\end{proposition}

\begin{proof}
This follows directly from the formula 
$L_n=\alpha^n + (-1)^n \alpha^{-n}.$
\end{proof}

\subsection{Linear forms in logarithms and continued fractions}

In order to prove our main result, we have to use a Baker-type lower bound several times for a non-zero linear forms of logarithms in algebraic numbers. There are many of these methods in the literature like that of Baker and Wüstholz in \cite{1}. We recall the result of Bugeaud, Mignotte, and Siksek which is a modified version of the result of Matveev \cite{8}. With the notation of section $2$, Laurent, Mignotte, and Nesterenko \cite{7} proved the following theorem:

\begin{theorem}\label{theo2}

Let $\gamma_1$, $ \gamma_2$ be two non-zero algebraic numbers, and let $\log \gamma_1$ and $\log \gamma_2$ be any determination of their logarithms. Put $D=[\mathbb{Q}(\gamma_1,\gamma_2):\mathbb{Q}]/[\mathbb{R}(\gamma_1,\gamma_2):\mathbb{R}]$, and 

$$
\Gamma:=b_2\log \gamma_2-b_1\log \gamma_1,
$$
 where $b_1$ and $b_2$ are positive integers. Further, let $A_1, A_2$ be real numbers $>1$ such that 
 
$$
\log A_i\geq\max \left\lbrace h(\gamma_i),\dfrac{|\log \gamma_i|}{D},\dfrac{1}{D}\right\rbrace,\quad  (i=1,2).
$$

 Then, assuming that $\gamma_1$ and $\gamma_2$ are mutiplicatively independent, we have

$$
\log |\Gamma|> -30.9\cdot D^4\left(\max \left\lbrace\log b',\dfrac{21}{D},\dfrac{1}{2}\right\rbrace\right)^2\log A_1\cdot \log A_2,
$$
 where 
$$
b'=\dfrac{b_1}{D\log A_2}+\dfrac{b_2}{D\log A_1}.
$$

\end{theorem}

We shall also need the following theorem due to Matveev, Lemma due to Dujella and Peth\H o and Lemma due to Legendre \cite{8,6}.

\begin{theorem} (Matveev \cite{8})\label{theo3}

 Let $n\geq 1$ an integer. Let $\mathbb{L}$ a field of algebraic number of degree $D$. Let $\eta_1$, \dots, $\eta_l$ non-zero elements of \ $\mathbb{L}$ and let 
$b_1$, $b_2$, \dots, $b_l$ integers, 
$$
B:=\max\{|b_1|,...,|b_l|\},
$$ 

and
\begin{equation*}
\Lambda:=\eta_1^{b_1}\cdots\eta_l^{b_l}-1=\left(\prod\limits_{i=1}^l \eta_i^{b_i}\right)-1.
\end{equation*}
Let $A_1$, \dots, $A_l$ reals numbers such that 
\begin{equation*}
A_j\geq \max\{Dh(\eta_j),|\log (\eta_j)|,0.16\}, 1\leq j\leq l.
\end{equation*}
Assume that  $\Lambda\neq 0$, So we have
\begin{equation*}
\log|\Lambda|>-3\times 30^{l+4}\times (l+1)^{5.5}\times d^2 \times A_1...A_l(1+\log D)(1+\log nB)
\end{equation*}
Further, if\  $\mathbb{L}$ is real, then 
\begin{equation*}
\log|\Lambda|>-1.4\times 30^{l+3}\times (l)^{4.5}\times d^2 \times A_1...A_l(1+\log D)(1+\log B).
\end{equation*}

\end{theorem}

During our calculations, we get upper bounds on our variables which are too large, so we have to reduce them. To do this, we use some results from the theory of continued fractions. In particular, for a non-homogeneous linear form with two integer variables, we use a slight variation of a result due to Dujella and Peth\H o, (1998) which is in itself a generalization of the result of Baker and Davemport \cite{2}.

For a real number $X$, we write $\norm{X} :=\min\{\mid X-n\mid :n\in \mathbb{Z}\}$ for the distance of $X$ to the nearest integer.

\begin{lemma}\label{lem2}(Dujella and Peth\H o, \cite{6}) 

Let $M$ a positive integer, let $p/q$ the convergent of the continued fraction expansion of $\kappa$ such that $q>6M$ and let  $A$, $B$, $\mu$ real numbers such that  $A>0$ and $B>1$. Let $\varepsilon:=\norm{\mu q}-M\norm{\kappa q}$.\\ If $\varepsilon>0$ then  there is no solution of the inequality 
\begin{equation*}
0<m\kappa -n+\mu< AB^{-m}
\end{equation*}
 in integers $m$ and $n$ with 
$$
\dfrac{\log (Aq/\varepsilon)}{\log B}\leqslant m \leqslant M.
$$

\end{lemma}

\begin{lemma} \label{lem1}(Legendre) 

Let $\tau$ real number such that $x$, $y$ are integers such that 
\begin{equation*}
\left| \tau - \dfrac{x}{y}\right|<\dfrac{1}{2y^{2}}.
\end{equation*}
  then $\dfrac{x}{y}=\dfrac{p_{k}}{q_{k}}$  is the convergence of $\tau$. 
  
\end{lemma}
Further, 
  \begin{equation*}
\left| \tau - \dfrac{x}{y}\right|>\dfrac{1}{(q_{k+1}+2)y^{2}}.
\end{equation*}

\section{Main result}
Our main result can be stated in the following theorem.

\begin{theorem}\label{theo}

	The only solutions $(n,m,a)$ of the exponential Diophantine equation 
	
	$L_n+L_m=3^a$ in nonnegative integers $n \geq m$ and $a$, are : $(1,0,1)\quad \text{and}\quad (4,0,2)$
 $$
 \text{i.e}\quad L_1+L_0=3, \quad \text{and}\quad L_4+L_0=9.
 $$ 
	
\end{theorem}

\begin{proof}

 First, we study the case $n=m$, next we assume $n>m$ and study the case $n\leq 200$ with \textit{SageMath} in the range $0\leq m< n\leq 200$ and finally we study the case $n> 200$. Assume throughout that equation (\ref{eq:1}) holds. First of all, observe that if $n = m$, then the original equation (\ref{eq:1}) becomes $$L_n=\dfrac{3^a}{2}.$$ This equation has no solution because, $\forall n> 0, L_n\in \mathbb{Z}.$ So from now, we assume $n>m$.\\ If $n\leq 200$, the  search with \textit{SageMath} in the range $0\leq m<n\leq 200$ gives the solutions $(n,m,a)\in \left\lbrace (1,0,1), (4,0,2)\right\rbrace$. Now for the rest of the paper, we assume that $n> 200$ . Let first get a relation between $a$ and $n$ which is important for our purpose. 
Combining (\ref{eq:1}) and the right inequality of (\ref{eq:2}), we get:
$$3^a=L_n+L_m\leq 2\alpha^n+2\alpha^m< 2^{n+1}+2^{m+1}=2^{n+1}(1+2^{n-m})\leq 2^{n+1}(1+1/2)<2^{n+2}.$$ Taking $\log$  both sides, we obtain $$a\log 3\leq (n+2)\log 2\Longrightarrow a\leq(n+2)c_1\quad \text{where}\quad c_1=\frac{\log 2}{\log 3}.$$ Rewriting equation (\ref{eq:1}) as:
$$
L_n+L_m=\alpha^n+\beta^n+L_m=3^a\Longrightarrow \alpha^n-3^a=-\beta^n-L_m.
$$ 
Taking absolute value both sides, we get
$$
|\alpha^n-3^a|=|\beta^n+L_m|\leq |\beta|^n+L_m< \frac{1}{2}+2\alpha^m \quad\because |\beta|^n<\frac{1}{2},\quad\text{and}\quad L_m< 2\alpha^m.
$$
Dividing both sides by $\alpha^n$ and considering that $n>m$, we get:
$$
\left|1-\alpha^{-n}\cdot 3^a\right|<\dfrac{\alpha^{-n}}{2}+2\alpha^{m-n}<\dfrac{1}{\alpha^{n-m}}+\dfrac{2}{\alpha^{n-m}}\quad \because \frac{1}{2\alpha^n}< \frac{1}{\alpha^{n-m}};\quad n>m
$$
Hence 
\begin{equation}\label{eq:3}
\left|1-\alpha^{-n}\cdot 3^a\right|<\dfrac{3}{\alpha^{n-m}}
\end{equation} 

Let's take 
$$
\gamma_1:=\alpha,\quad \gamma_2:=3,\quad b_1:=n,\quad b_2:=a,\quad \Gamma:= a\log 3-n\log \alpha
$$ 
in order to apply Theorem \ref{theo2}. Therefore equation (\ref{eq:3}) can be rewritten as:
 \begin{equation}\label{eq:4}
\left|1-e^{\Gamma}\right|< \dfrac{3}{\alpha^{n-m}}\quad \text{where}\quad e^{\Gamma}=\alpha^{-n}3^a.
\end{equation}

Since $\mathbb{Q}(\sqrt{5})$ is the algebraic number field containing $\gamma_1,\gamma_2$; so we can take $D:=2$. Using equation (\ref{eq:1}) and Binet formula for Lucas sequence, we have : $$\alpha^n=L_n-\beta^n<L_n+1\leq L_n+L_m=3^a$$ which implies $1<3^a\alpha^{-n}$ and so $\Gamma> 0$. Combining this with (\ref{eq:4}), we get  

\begin{equation}\label{eq:5}
0<\Gamma<\dfrac{3}{\alpha^{n-m}}
\end{equation}
where we used the fact that $x\leq e^x-1,\quad \forall x\in \mathbb{R}.$ Applying $\log$ on right and left hand side of (\ref{eq:5}), we get 
\begin{equation}\label{eq:6}
\log \Gamma < \log 3-(n-m)\log \alpha.
\end{equation}
Logarithm height of $\gamma_1$ and $\gamma_2$ are:

$h(\gamma_1)=\dfrac{1}{2}\log \alpha= 0.2406\cdots,$ 
$h(\gamma_2)=\log 3= 1.09862\cdots,$ thus we can choose
$$
\log A_1:= 0.5 \quad \text{and}\quad \log A_2 := 1.1.  
$$

Finally, by recalling that $a\leq (n+2)c_1; \quad c_1=0.63093,$ we get :
$$
b':=\dfrac{b_1}{D\log A_2}+\dfrac{b_2}{D\log A_1}=\dfrac{n}{2.2}+a=0.45n+a< 0.45n+(n+2)c_1< 2n.
$$

It is easy to see that $\alpha$ and $3$ are multiplicatively independent. Then by Theorem \ref{theo2}, we have 
$$
\log \Gamma \geq -30.9 \cdot 2^4\left(\max \left\lbrace \log (2n),\dfrac{21}{2},\dfrac{1}{2}\right\rbrace\right)^2\cdot 0.5\cdot 1.1
$$
\begin{equation}\label{eq:7}
 \log \Gamma> -272\left(\max \left\lbrace \log (2n),\dfrac{21}{2},\dfrac{1}{2}\right\rbrace\right)^2. 
\end{equation}
Combining (\ref{eq:6}) and (\ref{eq:7}), we obtain the following important result 
\begin{equation}\label{eq:8}
(n-m)\log \alpha < 276 \left(\max \left\lbrace \log (2n),\dfrac{21}{2},\dfrac{1}{2}\right\rbrace\right)^2.
\end{equation}
Let us find a second linear form in logarithm. For this, we rewrite (\ref{eq:1}) as follows:

$$
\alpha^n(1+\alpha^{n-m})-3^a=-\beta^n-\beta^m.
$$
Taking absolute values in the above relation, we get $$|\alpha^n(1+\alpha^{m-n})-3^a|< 2, \quad \beta=(1-\sqrt{5})/2,\quad  |\beta|^n<1 \quad \text{and}\quad  |\beta|^m<1; \forall n> 200,\quad m\geq 0.$$ Dividing both sides of the above inequality by $\alpha^n(1+\alpha^{m-n})$, we obtain
 
\begin{equation}\label{eq:9}
\left|1-3^a\alpha^{-n}(1+\alpha^{m-n})^{-1}\right|< \dfrac{2}{\alpha^n} \quad \text{i.e}\quad |\Lambda|<\dfrac{2}{\alpha^n}.
\end{equation}

All the conditions are now met to apply a Matveev's theorem (Theorem \ref{theo3}).
\begin{itemize}
\item[•] \text{Data:}

$$
t:=3;\quad \gamma_1:=3; \quad \gamma_2:=\alpha; \quad \gamma_3:=1+\alpha^{m-n}
$$

$$
b_1:=a;\quad\quad  b_2:=-n,\quad b_3=-1.
$$
As before, $\mathbb{K}=\mathbb{Q}(\sqrt{5})$ contains $\gamma_1, \gamma_2,\gamma_3$ and has $D:=[\mathbb{K}:\mathbb{Q}]=2.$
Before continuing with the calculations, let's check whether $\Lambda\neq 0$.

$\Lambda\neq 0$ comes from the fact that if it was zero, we would have 

\begin{equation}\label{eq:10}
3^a=\alpha^n+\alpha^m
\end{equation}

Taking the conjugate of the above relation in $\mathbb{Q}(\sqrt{5})$, we get :

\begin{equation}\label{eq:11}
3^a=\beta^n+\beta^m.
\end{equation}

Combining (\ref{eq:10}) and (\ref{eq:11}), we get :

$$\alpha^n<\alpha^n+\alpha^m=|\beta^n+\beta^m|\leq |\beta|^n+|\beta|^m< 2.$$
Recall that $n>200$. This relation is impossible for $n>200$. Hence $\Lambda \neq 0.$
\item[•] \textbf{Calculation of $h(\gamma_3)$}

Let us now estimate $h(\gamma_3)$ where $\gamma_3= 1+\alpha^{m-n}$
$$
\gamma_3= 1+\alpha^{m-n}< 2\quad \text{and}\quad \gamma^{-1}=\dfrac{1}{1+\alpha^{m-n}}<1
$$
so $ |\log \gamma_3|<1.$ Notice that 
$$
h(\gamma_3)\leq|m-n|\left(\dfrac{\log \alpha}{2}\right)+\log 2=\log 2+(n-m)\left(\dfrac{\log \alpha}{2}\right).
$$
\item[•] \text{The calculation of $A_1$ and $A_2$ gives :}

$$
A_1:=2.2 
$$
and $$
A_2:=0.5 
$$
and we can take  
$$
A_3:=2+(n-m)\log \alpha \quad \text{since}\quad h(\gamma_3):= \log 2+(n-m)\left(\dfrac{\log \alpha}{2}\right)
$$
\item[•] \textbf{Calculation of $B$}

Since $a< (n+2)c_1$, it follows that, $B=\max \{1,n,a\}$. Thus we can take $B=n+1$. 
\end{itemize}
The Matveev's theorem gives the lower bound on the left hand side of (\ref{eq:9}) by replacing the data. We get :
$$
\exp \left(-C(1+\log(n+1))\cdot 2.2\cdot0.5\cdot(2+(n-m)\log \alpha)\right)
$$
 where 
$$
C:=1.4\cdot30^6\cdot 3^{4.5}\cdot 2^2(1+\log 2)< 9.7\times 10^{11}.
$$
Replacing in equation (\ref{eq:9}), we get:
$$
\exp \left(-C(1+\log(n+1))\cdot 2.2\cdot0.5\cdot(2+(n-m)\log \alpha)\right)<|\Lambda|<\dfrac{2}{\alpha^n}
$$ which leads to 

\begin{equation*}
n\log \alpha-\log 2<C((1+\log(n+1))\cdot 1.1\cdot(2+(n-m)\log \alpha)<2C\log n\cdot 1.1\cdot(2+(n-m)\log \alpha)
\end{equation*}
then \begin{equation}\label{eq:12}
n\log \alpha-\log 2< 1.26\times 10^{12}\log n\cdot(2+(n-m)\log \alpha)
\end{equation}

where we used inequality $1+\log(n+1)<2\log n$, which holds for $n>200$.

Now, using (\ref{eq:8}) in the right term of the above inequality (\ref{eq:12}) and doing the related calculations, we get 

\begin{equation}\label{eq:13}
n< 7.3\times 10^{14}\log n \left(\max\left\lbrace \log(2n),\dfrac{21}{2}\right\rbrace\right)^2.
\end{equation}

If $\max\{\log(2n),21/2\}=21/2,$ it follows from (\ref{eq:13}) that $n<8.04825 \cdot 10^{16}\log n \Longrightarrow n<3.5\cdot 10^{18}.$ On the other hand, if  $\max\{\log(2n),21/2\}=\log (2n)$, then from (\ref{eq:13}), we get $n< 7.3\cdot 10^{14}\log n\log^2(2n)$ and so $n<7.2 \cdot 10^{19}$. We can easily see that for the two possible values of $\max\{\log(2n),21/2\}$,  $n<7.2\cdot10^{19}$.

All the calculations done so far can be summarized in the following lemma.

\begin{lemma}\label{lem3} 

If $(n,m,a)$ is a solution in positive integers of (\ref{eq:1}) with conditions $n>m$ and $n>200$, then inequalities 
$$
a\leq n+2< 1.2\times 10^{20}	\quad \text{hold.}
$$ 
\end{lemma}

\section{Reducing of the bound on $n$}

Dividing across inequality  (\ref{eq:5}) : $0<a\log 3-n\log \alpha<\dfrac{3}{\alpha^{n-m}}$ by $\log \alpha$, we get 

\begin{equation}\label{eq:14}
0<a\gamma -n<\dfrac{7}{\alpha^{n-m}};\quad \text{where}\quad \gamma:=\dfrac{\log 3}{\log \alpha}.
\end{equation}

The continued fraction of the irrational number $\gamma$ is : $$[a_0,a_1,a_2,......]=[1,2,3,1,1,2,3,2,4,2,1,11,2,1,11,......]$$ and let denote $p_k/q_k$ its convergent. An inspection using \textit{SageMath} gives the following inequality
$$
4977896525362041575=q_{41}<1.2\times 10^{20}<q_{42}=805929983250536127817.
$$
Furthermore, $a_M:=\max \left\lbrace a_i| i=0,1,...,42\right\rbrace=161$
Now applying Lemma \ref{lem1} and properties of continued fractions, we obtain

\begin{equation}\label{eq:15}
|a\gamma-n|>\dfrac{1}{(a_M+2)a}.
\end{equation}

Combining equation (\ref{eq:14}) and (\ref{eq:15}), we get
 $$
 \dfrac{1}{(a_M+2)a}<|a\gamma-n|<\dfrac{7}{\alpha^{n-m}}\Longrightarrow \dfrac{1}{(a_M+2)a}<\dfrac{7}{\alpha^{n-m}}\Longrightarrow \alpha^{n-m}<7\cdot (161+2)a<1.3692 \cdot 10^{23}.
 $$
Applying $\log$ above and divide by $\log\alpha$, we get : 
$$
(n-m)\leq \dfrac{\log\left(7\cdot 163\cdot a\right)}{\log \alpha}< 111.
$$
To improve the upper bound on $n$, let consider 
\begin{equation}\label{eq:16}
z:=a\log 3-n\log \alpha-\log \rho (u)\quad \text{where}\quad \rho=1+\alpha^{-u}.
\end{equation}

From (\ref{eq:9}), we have 
\begin{equation}\label{eq:17}
|1-e^z|<\dfrac{2}{\alpha^n}.
\end{equation}
Since $\Lambda \neq 0$, then $z\neq 0$. Two cases arise : $z<0$ and $z> 0$. For each case, we will apply Lemma \ref{lem2}.

\begin{itemize}
\item[•] \textbf{Case 1 : $z> 0$}

 From (\ref{eq:17}), we obtain $0<z\leq e^z-1<\dfrac{2}{\alpha^n}.$
Replacing (\ref{eq:16}) in the above inequality, we get:
$$
0< a\log 3-n\log \alpha-\log \rho (n-m)\leq 3^{a}\alpha^{-n}\rho (n-m)^{-1}-1<2\alpha^{-n}$$ hence $$0< a\log 3-n\log \alpha-\log \rho (n-m)<2\alpha^{-n}
$$
and by diving above inequality by $\log \alpha$
\begin{equation}\label{eq:18}
 0< a\left(\dfrac{\log 3}{\log \alpha}\right)-n-\dfrac{\log \rho (n-m)}{\log \alpha}<5\cdot\alpha^{-n}.
\end{equation}
Taking, $\gamma:=\dfrac{\log 3}{\log \alpha},\quad \mu:=-\dfrac{\log \rho (n-m)}{\log \alpha},\quad A:=5,\quad B:=\alpha$, inequality (\ref{eq:18}) becomes
$$
0<a\gamma-n+\mu<AB^{-n}.
$$ 
Since $\gamma$ is irrational, we are now ready to apply lemma \ref{lem2} of Dujella and Pethö on (\ref{eq:18}) for $n-m \in \{1,...,111\}$.

Since $a\leq 1.2\times 10^{20}$ from lemma \ref{lem3}, we can take $M=1.2\times 10^{20},$ and we get
 $$
 n<\dfrac{\log(Aq/\varepsilon)}{\log B}\quad \text{where} \quad q> 6M
 $$ 
and $q$ is the denominator of the convergent of the irrational number $\gamma$ such that $\varepsilon:=||\mu q||-M||\gamma q||>0$. With the help of \textit{SageMath}, with conditions $z>0, \quad \text{and} \quad (n,m,a)$ a possible zero of (\ref{eq:1}), we get $n< 112$ which contradicts our assumption $n>200$. Then it is false.

\item[•]\textbf{Case 2 : $z< 0$}

Since $n>200$, then $\frac{2}{\alpha^n}<\frac{1}{2}$. Hence (\ref{eq:17}) implies that $|1-e^{|z|}|< 2$. Also, since $z<0$, we have $$0<|z|\leq e^{|z|}-1= e^{|z|}| e^{|z|}-1|<\dfrac{4}{\alpha^n}.$$ Replacing (\ref{eq:16}) in the above inequality and dividing by $\log 3$, we get:
\begin{equation}\label{eq:19}
0<n\left(\dfrac{\log \alpha}{\log 3}\right)-a+\dfrac{\rho(n-m)}{\log 3}< \dfrac{4}{\log 3}\cdot\alpha^{-n}< 4\cdot\alpha^{-n}
\end{equation}

\end{itemize}
In order to apply lemma \ref{lem3} on (\ref{eq:19}) for $n-m\in \{1,2,...,111\}$, let's take again $M=1.2\times 10^{20}$. With the help of \textit{SageMath}, with conditions $z<0, \quad \text{and} \quad (n,m,a)$ a possible zero of (\ref{eq:1}), we get $n< 111$ which contradicts our assumption $n>200$. Then it is false.

 This completes the proof of our main result (Theorem \ref{theo}).
\end{proof}

\section*{Acknowledgments}
The authors thank the professor Maurice Mignotte for his remarks and diponibility.

\vspace{2cm}
\begin{multicols}{2}
\begin{flushleft}
\begin{center}
\textbf{Pagdame TIEBEKABE}\\
\small{
Université Cheikh Anta Diop (UCAD),}
\footnotesize{Laboratoire d'Algèbre, de Cryptologie,
de Géométrie Algébrique et Applications (LACGAA)}\\ 
\sc{Dakar, Sénégal}
\end{center} 
\end{flushleft}
\columnbreak
\begin{flushright}
\begin{center}
\textbf{Ismaïla DIOUF}\\
\small{
Université Cheikh Anta Diop (UCAD),}
\footnotesize{ Laboratoire d'Algèbre, de Cryptologie,
 de Géométrie Algébrique et Applications (LACGAA)}\\
\sc{Dakar, Sénégal}
\end{center}
\end{flushright}
\end{multicols}
 \end{document}